\documentclass[11pt]{article}

\usepackage{amssymb}
\usepackage[intlimits]{amsmath}
\usepackage{amsfonts}
\usepackage{amsthm}
\usepackage{cite}
\usepackage{mathptmx}

\usepackage{hyperref}

\usepackage[a4paper]{geometry}

\let\oldsqrt\sqrt
\def\sqrt{\mathpalette\DHLhksqrt}
\def\DHLhksqrt#1#2{%
\setbox0=\hbox{$#1\oldsqrt{#2\,}$}\dimen0=\ht0
\advance\dimen0-0.2\ht0
\setbox2=\hbox{\vrule height\ht0 depth -\dimen0}%
{\box0\lower0.4pt\box2}}

\allowdisplaybreaks

\newcommand{\R}{\mathbb{R}} 
\newcommand{\N}{\mathbb{N}} 
\newcommand{\dist}{\textnormal{dist}} 
\newcommand{\diam}{\textnormal{diam}} 
\newcommand{\supp}{\textnormal{supp}} 
\DeclareMathOperator{\essinf}{essinf} 

\renewcommand{\phi}{\varphi}

\newcommand{\cC}{{\mathcal C}}
\newcommand{\cD}{{\mathcal D}}

\newcommand{\cF}{{\mathcal F}}

\newcommand{\cH}{{\mathcal H}}

\newcommand{\cJ}{{\mathcal J}}

\newcommand{\cV}{{\mathcal V}}

\theoremstyle{definition}
\newtheorem{defi}{Definition}[section]
\newtheorem{bem}[defi]{Remark}
\newtheorem{bsp}[defi]{Example}

\theoremstyle{plain} 
\newtheorem{satz}[defi]{Theorem}

\newtheorem{prop}[defi]{Proposition}
\newtheorem{lemma}[defi]{Lemma}

\theoremstyle{definition}

\numberwithin{equation}{section}
\title{Symmetry of solutions to nonlocal nonlinear boundary value problems in radial sets}
\author{Sven Jarohs\footnote{Goethe-Universit\"at, Frankfurt, jarohs@math.uni-frankfurt.de.}
}
\date{\today}

\begin{document}
\maketitle
\begin{abstract}
For open radial sets $\Omega\subset \R^N$, $N\geq 2$ we consider the nonlinear problem 
\[
 (P)\qquad\left\{\begin{aligned}
  I u&=f(|x|,u)&& \text{ in $\Omega$,}\\
	u&\equiv 0 &&\text{ on $\R^{N}\setminus \Omega$,}\\
\lim_{|x|\to\infty}u(x)&=0,&&
 \end{aligned}\right.
\]
where $I$ is a nonlocal operator and $f$ is a nonlinearity. Under mild symmetry and monotonicity assumptions on $I$, $f$ and $\Omega$ we show that any continuous bounded solution of $(P)$ is axial symmetric once it satisfies a simple reflection inequality with respect to a hyperplane. In the special case where $f$ does not depend on $|x|$, we show that any nonnegative nontrivial continuous bounded solution of $(P)$ in $\R^N$ is radially symmetric (up to translation) and strictly decreasing in its radial direction. Our proves rely on different variants of maximum principles for antisymmetric supersolutions which can be seen as extensions of the results in \cite{JW14}. As an application, we prove an axial symmetry result for minimizers of an energy functional associated to $(P)$.
\end{abstract}

{\footnotesize
\begin{center}
\end{center}
\begin{center}
\end{center}
}
\section{Introduction}
In this work we study a class of nonlocal and nonlinear boundary value problems in an open set $\Omega\subset \R^N$. To be precise, we study for $s\in(0,1)$ bounded continuous solutions of the problem  
\[
 (P)\qquad\left\{\begin{aligned}
  (-\Delta)^s u&=f(|x|,u)&& \text{ in $\Omega$;}\\
	u&\equiv 0&& \text{ on $\R^{N}\setminus \Omega$;}\\
\lim_{|x|\to\infty}u(x)&=0,&&
 \end{aligned}\right.
\]
with a nonlinearity $f: [0,\infty)\times \R\to\R$. Here $(-\Delta)^s$ is the fractional Laplacian which is usually defined via the Fourier transform as
\[
 \cF((-\Delta)^su)(\xi)=|\xi|^{2s}\cF(u)(\xi) \quad\text{  for $\xi \in \R^N$ and $u\in \cC^2_c(\R^N)$.}
\]
Note that $(-\Delta)^su$ can be represented via the following principle value integral
\[
 (-\Delta)^su(x)=c_{N,s} P.V. \int_{\R^{N}} \frac{u(x)-u(y)}{|x-y|^{N+2s}}\ dy:= c_{N,s}\lim_{\epsilon\to0}\int_{|x-y|>\epsilon}\frac{u(x)-u(y)}{|x-y|^{N+2s}}\ dy,
\]
where $c_{N,s}= s(1-s)\pi^{-N/2}4^{s}\frac{\Gamma(\frac{N+2s}{2})}{\Gamma(1-s)}$ is a normalization constant. In the following we will work with a weak formulation of solutions, i.e. a function $u\in \cH^s_0(\Omega):=\{u\in H^s(\R^N)\;:\; u\equiv 0\; \text{on $\R^N\setminus \Omega$}\}$ is called a solution of $(P)$, if for all $\varphi\in\cH^s_0(\Omega)$ with compact support in $\R^N$ the integral $\int_{\R^{N}}f(|x|,u(x))\varphi(x)\ dx$ exists and we have
\[
 \frac{c_{N,s}}{2}\int_{\R^N}\int_{\R^{N}}\frac{(u(x)-u(y))(\varphi(x)-\varphi(y))}{|x-y|^{N+2s}}\ dxdy=\int_{\Omega}f(|x|,u(x))\varphi(x)\ dx.
\]
Here $H^s(\R^N)=\{u\in L^2(\R^N)\;:\; \int_{\R^N}\int_{\R^{N}}\frac{(u(x)-u(y))^2}{|x-y|^{N+2s}}\ dxdy<\infty\}$ is the fractional Sobolev space of order $s$ (see \cite{NPV11} and the references in there for more details on fractional Sobolev spaces).\\

Problem $(P)$ has been studied extensively in recent years (see e.g. \cite{BB00,BLW05,Chen_Li_Ou,FW13-2,Chen_Song,CS07,CS,FLS13,JW13}). In particular, if $\Omega$ is a ball and $f$ satisfies some monotonicity assumptions in $r$, there have been several results which prove that in this case nonnegative bounded solutions of $(P)$ are radial symmetric (see e.g. \cite{BLW05,JW13,sciunzi}). The case where $\Omega$ is the exterior of a ball, radial symmetry for a related problem with a different boundary condition has been provided in \cite{SV14}. The proofs in these works are based on different variants of the moving plane method.\\

In this work we will analyze problem $(P)$ in radial sets $\Omega$, e.g. for $R>r>0$ we consider $\Omega=B_R(0)$, $\Omega=B_R(0)\setminus B_r(0)$ $\Omega=\R^N$, $\Omega=\R^N\setminus B_R(0)$, $\Omega=\R^N\setminus B_R(0)\cup B_r(0)$, etc. Moreover, we will allow the function $u$ to change sign. An approach with the moving plane method as in the above works is not possible and in general radial symmetry cannot be expected in this case. We will consider a particular kind of axial symmetry called \textit{foliated Schwarz symmetry} (see \cite{SW03}; we also refer to the general survey -- in particular Section 2.3 -- in \cite{T})\\

Let $D\subset\R^N$, $N\geq 2$ be a radial domain, $p\in S^{N-1}:=\{x\in \R^N\;:\; |x|=1\}$. A function $u:D\to \R$ is called \textit{foliated Schwarz symmetric with respect to $p$ in $D$}, if for every $r>0$ with $re_1\in D$ and $c\in \R$, the restricted superlevel set $\{x\in rS^{N-1}\;:\; u(x)\geq c\}$ is equal to $rS^{N-1}$ or a geodesic ball in the sphere $rS^{N-1}$ centered at $rp$.\\
We simply call $u$ \textit{foliated Schwarz symmetric}, if $u$ has this property for some unit vector $p\in \R^N$.\\

We will give an equivalent definition in Section \ref{foliatedsymmgoal} (see also \cite[Proposition 3.3]{SW13}) below, which we will use in our proofs. Note that if $u:\R^N\to \R$ is such that $u|_D$ is foliated Schwarz symmetric with respect to some $p$ for some radial set $D\subset \R^N$, then $u1_D$ is axially symmetric with respect to the axis $\R\cdot p$ and nonincreasing in the polar angel $\theta=\arccos (\frac{x}{|x|}\cdot p)$.

In order to prove foliated Schwarz symmetry of solutions, we will elaborate a rotating plane method for nonlocal boundary value problems. This method can be seen as a variant of the moving plane method and is used in the local case $s=1$ in \cite{P01,PW07,GPW10,DP13} to prove axial symmetry for solutions with low Morse index. Later in \cite{SW13} this method has been used to prove axial symmetry for solutions of a related local time depended problem. To start the rotating plane method one usually assumes that $u$ has the following property: 
\begin{enumerate}\label{u1}
\item[(U1)] There is $e_1\in S^{N-1}$ such that $u\geq u\circ Q_{H_{e_1}}$ with $u\not\equiv u\circ Q_{H_{e_1}}$ on $H_{e_1}:=\{x\in \R^N\;:\; x\cdot e_1>0\}$, where $Q_{H_{e_1}}$ denotes the reflection at $\partial H_{e_1}$.
\end{enumerate}

In case $\Omega$ is bounded, we have the following result.

\begin{satz}\label{result2}
Let $\Omega\subset\R^N$, $N\geq 2$ be a radial bounded set. Furthermore, assume $f$ satisfies
\begin{enumerate}
\item[(F1)] For every $K>0$ there is $L=L(K)>0$ such that
\[
\sup_{r\in[0,\infty)}|f(r,u)-f(r,v)|\leq L|u-v|\qquad \text{ for all $u,v\in[-K,K]$.}
\] 
\end{enumerate}
 Then any continuous bounded solution $u$ of $(P)$ which satisfies (U1) is foliated Schwarz symmetric.
\end{satz}

If $\Omega$ is unbounded we will need the following assumption on the decay of $u$:
\begin{enumerate}\label{u2}
 \item[(U2)] $\lim\limits_{|x|\to\infty}u(x)=0$.
\end{enumerate}

\begin{satz}\label{result3}
Let $\Omega\subset\R^N$, $N\geq 2$ be a radial set and let $f$ satisfy (F1) and 
\begin{enumerate}
\item[(F2)] $f(r,0)=0$ for all $r\in[0,\infty)$ and there is $\delta>0$ such that
\[
\sup_{r\in[0,\infty)}\frac{f(r,u)}{u}\leq 0\qquad\text{ for all $u\in [-\delta,\delta]$.}
\]
\end{enumerate}
Then any continuous bounded solution $u$ of $(P)$ which satisfies (U1) and (U2) is foliated Schwarz symmetric.
\end{satz}

\begin{bem}
 We emphasize that $\Omega$ does not need to be connected in Theorems \ref{result2} and \ref{result3}.
\end{bem}

As an application we will analyze global minimizers of functionals of the form
\[
K:\ \cH^s_0(\Omega)\to \R,\qquad K[u]=\frac{c_{N,s}}{4}\int_{\R^N}\int_{\R^{N}}\frac{(u(x)-u(y))^2}{|x-y|^{N+2s}}\ dxdy-\int_{\Omega} F(|x|,u(x))\ dx.
\]
Here $\Omega\subset \R^N$, $N\geq 2$ is a radial open set and
\[
 F(r,u)=\int_0^{u} f(r,\tau)\ d\tau\qquad\text{ for $r\geq0$, $u\in \R$,}
\]
where $f:[0,\infty)\times \R\to \R$ satisfies (F1), (F2) and for some $a_1,a_2$ and $q\in[2,\frac{2N}{N-2s}]$ we have
\[
 |f(r,u)|\leq a_1|u|+ a_2|u|^{q-1}\qquad\text{ for all $r\geq 0$, $u\in \R$.}
\]

\begin{satz}\label{result4}
 Let $\Omega$ and $F$ be given as above and assume $u\in \cH^s_0(\Omega)$ satisfies either
\begin{enumerate}
 \item $u$ is a continuous bounded global minimizer of $K$ or
\item $u$ is a continuous bounded minimizer of $K$ subject to $\|u\|^q=1$, where $q$ is as for $F$.
\end{enumerate}
If $u$ satisfies additionally (U2), then $u$ is foliated Schwarz symmetric.
\end{satz}

\begin{bem}
There have been several results on the analysis of minimizers involving nonlocal equations, we refer e.g. to \cite{SV12,SV13,B14,SV15}. However, up to the authors knowledge there have been no symmetry results such as the one above.
\end{bem}

To prove our results, we will work only with the bilinearform associated to the fractional Laplacian. In this way, we will only need the monotonicity and symmetry properties of the kernel function $z\mapsto |z|^{-N-2s}$, $z\in \R^N\setminus\{0\}$ and in particular not its isotropy or its scaling laws. Hence our results extend to a more general class of nonlocal operators. In the spirit of \cite{FK12,FKV13, KM14b,SV13, JW14,J15} we will consider a more general class of nonlocal operators which includes the fractional Laplacian but also considers operators of e.g. zeroth order. To be precise, we will consider a self-adjoint nonnegative operator $I$ which is associated to the following nonlocal bilinear form
\begin{equation}\label{bilinearform}
\cJ(v,w)=\int_{\R^{N}}\int_{\R^N} (v(x)-v(y))(w(x)-w(y))k(|x-y|)\ dxdy \qquad v,w\in \cC^2_c(\R^N),
\end{equation}
where $k:(0,\infty)\to [0,\infty)$ is a decreasing function such that
\begin{enumerate}
\item[(k)] there is $r_0>0$ with $k|_{(0,r_0)}$ is strictly decreasing and $k$ satisfies
\[
\int_{0}^{\infty} \min\{1,r^2\} k(r)r^{N-1}\ dr<\infty\quad\text{ and }\quad \int_{0}^{\infty}k(r)r^{N-1}\ dr=\infty.
\]
\end{enumerate}
We note that under these assumptions the operator $I$ is represented for $\varphi\in \cC^2_c(\R^N)$, $x\in \R^N$ by
\begin{equation}\label{defiop}
[I\varphi](x)=P.V. \int_{\R^{N}} (\varphi(x)-\varphi(y))k(|x-y|)\ dy:= \lim_{\epsilon\to0}\int_{|x-y|>\epsilon}(\varphi(x)-\varphi(y))k(|x-y|)\ dy.
\end{equation}
With $I$ in place of $(-\Delta)^s$ in $(P)$, we then analyze symmetry properties of solutions of
\[
 (P')\qquad\left\{\begin{aligned}
  I u&=f(|x|,u)&& \text{ in $\Omega$;}\\
	u&\equiv 0&& \text{ on $\R^{N}\setminus \Omega$;}\\
\lim_{|x|\to\infty}u(x)&=0,&&
 \end{aligned}\right.
\]

Here we will again use a weak formulation of solutions, i.e. for $\Omega\subset\R^N$ open, denote
\begin{equation}\label{defispace}
\cD^{J}(\Omega):=\{u\in L^2(\R^{N})\;:\; \cJ(u,u)<\infty \;\text{and}\; u\equiv 0 \;\text{ on $\R^N\setminus \Omega$}\},
\end{equation}
which equipped with the scalar product
\[
\langle u,v\rangle_J:=\langle u,v\rangle_{L^2(\R^{N})}+\cJ(u,v)
\]
is a Hilbert space. A function $u\in \cD^{J}(\Omega)$ is called a \textit{solution of $(P')$} if for all $\varphi\in \cD^{J}(\Omega)$ with compact support in $\R^N$ the integral $\int_{\R^{N}}f(|x|,u(x))\varphi(x)\ dx$ exists and 
\[
\cJ(u,\varphi)=\int_{\Omega}f(|x|,u(x))\varphi(x)\ dx.
\]
The following examples satisfy (k):
\begin{bsp}\label{bsp1}
Let $k:(0,\infty)\to(0,\infty)$ be given for $r>0$ by
\begin{enumerate}
\item For $s\in(0,1)$ the function $k:(0,\infty)\to(0,\infty)$ given by $k(r)=c_{N,s}r^{-N-2s}$ for $r>0$ with $c_{N,s}$ as above satisfies (k).
\item $k(r)= 1_{[0,1]}(r)r^{-N}$ or $k(r)= -1_{[0,1]}(r)[\ln r]^{-N}$ satisfy (k). In particular, these give examples for operators of order zero.
\item $k(r)= \tilde{c}_{N,s}r^{-\frac{N+2s}{2}}K_{\frac{N+2s}{2}}(r)$, $s\in(0,1)$, where $K_\nu:(0,\infty)\to [0,\infty)$ denotes the modified Bessel function of second kind with order $\nu$ and $\tilde{c}_{N,s}=s(1-s)\pi^{-\frac{N}{2}}4^s \frac{2^{1-\frac{N+2s}{2}}}{\Gamma(2-s)}$ is a normalization constant. Note that for $u\in \cC^2_c(\R^N)$ the operator $(\textnormal{id}-\Delta)^s$ has the following integral representation (see e.g. \cite{FF14})
\[
\left[(\textnormal{id}-\Delta)^su\right](x) =P.V.\int_{\R^{N}}(u(x)-u(y))k(|x-y|)\ dy + u(x).
\]
Thus this operator is included in problem $(P')$ with $I=(\textnormal{id}-\Delta)^s-\textnormal{id}$ and $f$ replaced by $\tilde{f}(r,u)=-u+f(r,u)$ for $r\geq0$, $u\in \R$.
\end{enumerate}
\end{bsp}

Theorems \ref{result2} and \ref{result3} then extend to

\begin{satz}\label{schwarzsymm}
Let $\Omega\subset\R^N$, $N\geq 2$ be a bounded radial set. Moreover, if $\Omega$ is not connected, assume additionally that $k:(0,\infty)\to(0,\infty)$ is strictly decreasing. Let $f$ satisfies (F1). Then every bounded continuous solution $u$ of $(P')$ satisfying (U1) $u$ is foliated Schwarz symmetric.
\end{satz}

\begin{satz}\label{schwarzsymm2}
Let $\Omega\subset\R^N$, $N\geq 2$ be a unbounded radial set. Moreover, if $\Omega$ is not connected, assume additionally that $k:(0,\infty)\to(0,\infty)$ is strictly decreasing. Let $f$ satisfies (F1) and (F2). Then every bounded continuous solution $u$ of $(P)$ satisfying (U1) and (U2) is foliated Schwarz symmetric.
\end{satz}

\begin{bem}
\begin{enumerate}
\item  Note that Theorem \ref{result2} and \ref{result3} are special cases of Theorem \ref{schwarzsymm}, \ref{schwarzsymm2} resp. 
\item By a recent regularity result of Kassmann and Mimica \cite{KM14b} we have that a bounded solution $u\in \cD^J(\Omega)$ of $(P')$ satisfies $u\in C(\Omega)$ if (F1) is satisfied.
\end{enumerate}
\end{bem}

Finally, we will also consider the case of nonnegative bounded solutions of
\[
 (R)\qquad\left\{\begin{aligned}
  Iu&=f(u)&& \text{ in $\R^N$;}\\
\lim_{|x|\to\infty}u(x)&=0,&&
 \end{aligned}\right.
\]
where $I$ is as before the nonlocal operator associated to the bilinearform as in (\ref{bilinearform}) with a kernel function $k$ satisfying (k).

\begin{satz}\label{result1}
 Assume (k), $N\in \N$ and let $f:\R\to \R$ satisfy (F1) and (F2). Then every nonnegative bounded continuous solution $u \in \cD^{J}(\R^N)$ of $(R)$
is radial symmetric up to translation, i.e. there is $z_0\in \R^N$ such that $u(\cdot-z_0)$ is radially symmetric. Moreover, either $u \equiv 0$ in $\R^{N}$, or $u(\cdot-z_0)$ is strictly decreasing in its radial direction and therefore $u>0$ in $\R^N$.
\end{satz}

Radial symmetry for local equations via the maximum principle goes back to \cite{GNN79,BN91}. For equations involving the fractional Laplacian symmetry results have been shown in \cite{FW13-2,felmer-quaas-tan,Chen_Li_Ou}. Chen, Li and Ou \cite{Chen_Li_Ou} show radial symmetry for solutions of equations of type $(-\Delta)^su=u^{\frac{N+2s}{N-2s}}$ in $\R^N$ via the inverse of the fractional Laplacian using the moving plane method for integral equations. This method is generalized by Felmer, Quaas and Tan \cite{felmer-quaas-tan} to prove radial symmetry for positive solutions of equations of type $(-\Delta)^s+u=f(u)$ in $\R^N$. In \cite{MC08} the authors show radial symmetry for positive solutions in $L^q(\R^N)$ of equations of type $(\textnormal{id}-\Delta)^su=u^{\beta}$ in $\R^N$, $\beta>1$, if $q>\max\{\beta, \frac{N(\beta-1)}{2s}\}$. For this they use also the inverse operator. For classical positive solutions of equations of type $(-\Delta)^su=f(u)$ in $\R^N$ radial symmetry was also analyzed in \cite{FW13-2}.

As before, our proof relies only on monotonicity and symmetry properties of the kernel function $k$. In particular, the inverse operator is not needed for our arguments. Moreover, we note that our approach only requires the solution $u$ to be nonnegative.

The paper is organized as follows. In Section \ref{setup} we will collect basic statements on nonlocal bilinear forms which we will need for our proves. Section \ref{mp} is devoted to a linearized form of problem (P') based on the difference of the solution and its reflected counterpart w.r.t. some hyperplane. In particular, we will prove different variants of maximum principles involving antisymmetric functions. The results stated in this section can be seen as generalizations to the results in \cite{JW14}. Section \ref{foliatedsymmgoal} is devoted to the proves of our axial symmetry results and in Section \ref{functional} we will apply these results to prove axial symmetry of global minimizers. Finally, in Section \ref{mr} we will give the proof of Theorem \ref{result1}.

\section*{Acknowledgments} The author thanks J.~Van~Schaftingen for the referenced inequality (\ref{pol}) below and T.~Weth for discussions and helpful comments.

\section{Preliminaries}\label{setup}
We fix some notation. For subsets $D,U \subset \R^N$ we write $\dist(D,U):= \inf\{|x-y|\::\: x \in D,\, y \in U\}$.  If $D= \{x\}$ is a singleton, we write $\dist(x,U)$
in place of $\dist(\{x\},U)$. For $U\subset\R^{N}$ and $r>0$ we consider $B_{r}(U):=\{x\in\R^{N}\;:\; \dist(x,U)<r\}$, and we let, as usual  
 $B_r(x)=B_{r}(\{x\})$ be the open ball in $\R^{N}$ centered at $x \in \R^N$ with radius $r>0$. Moreover, we denote $S^{N-1}:=\partial B_1(0)$. For any subset $M \subset \R^N$, we denote by $1_M: \R^N \to \R$ the
characteristic function of $M$ and by $\diam(M)$ the diameter of $M$. If $M$ is measurable $|M|$ denotes the Lebesgue measure of $M$. Moreover, if $w: M \to \R$ is a function,  we let $w^+= \max\{w,0\}$ resp. $w^-=-\min\{w,0\}$ denote the positive and negative part of $w$, respectively. 

Throughout the remainder of this paper, we assume the decreasing function $k:(0,\infty)\to[0,\infty)$ satisfies (k). We let $\cJ$ be the corresponding quadratic form defined in (\ref{bilinearform}) and, for an open set $\Omega \subset \R^N$, we consider $\cD^{J}(\Omega)$ as in (\ref{defispace}). It follows from (k) that $k$ is positive on a set of positive measure.  Thus, by   
\cite[Lemma 2.7]{FKV13} we have
\begin{equation}
\label{l2-bound}
\Lambda_{1}(\Omega):=\inf_{u\in \cD^{J}(\Omega)}\frac{\cJ(u,u)}{\|u\|^2_{L^{2}(\Omega)}} \: >\:0 \qquad \text{for every open bounded set  $\Omega\subset\R^{N}$,}
\end{equation}
which amounts to a Poincar\'e-Friedrichs type inequality. In particular, $\cD^{J}(\Omega)$ is a Hilbert space with scalar product $\cJ$ for any open bounded set $\Omega$. We will need lower bounds for $\Lambda_1(\Omega)$ in the case where $|\Omega|$ is small. For this we set for $\Omega\subset \R^N$ open

\begin{equation}\label{kappadefi}
\kappa_{\Omega}:\R^N\to [0,\infty], \qquad \kappa_{\Omega}(x)=\int_{\R^N\setminus \Omega}k(|x-y|)\ dy.
\end{equation}
Note that we have for any $u:\R^N\to \R$ measurable with $\cJ(u,u)<\infty$ and $\supp(u)\subset \Omega$

\[
\cJ(u,u)\geq \int_{\Omega}u^2(x)\kappa_{\Omega}(x)\ dx.
\]
 
\begin{lemma}[see \cite{JW14}, Lemma 2.1]\label{3-mengen-k}
We have $\Lambda_1(\Omega)\geq \inf\limits_{x\in \Omega} \kappa_{\Omega}(x)$ and
\[
\lim_{r\to 0}\;\;\inf_{\substack{\Omega\subset \R^N\\ |\Omega|=r}}\;\; \inf_{x\in \Omega}\; \kappa_{\Omega}(x)=\infty.
\]
\end{lemma}

\begin{prop}[see \cite{JW14}, Proposition 2.3]
(i) We have $\cC^{0,1}_{c}(\R^{N}) \subset \cD^{J}(\R^N)$.\\[0.1cm]
(ii) Let $v \in \cC_c^2(\R^N)$. Then the principle value integral in (\ref{defiop}) exists for every $x \in \R^N$. Moreover, $I v \in L^\infty(\R^N)$,  and for every bounded open set $\Omega \subset \R^N$ and every $u \in \cD^{J}(\Omega)$ we have 
$$
\cJ(u,v)= \int_{\R^N} u(x) [Iv](x)\,dx.
$$
\end{prop}

Next, we wish to extend the definition of $\cJ(v,\phi)$ to more general pairs of functions $(v,\phi)$. In the following, for a measurable subset $U' \subset \R^N$, we define $\cV^{J}(U')$ as the space of all functions $v \in L^2(\R^N)+ L^{\infty}(\R^{N})$ such that 
\begin{equation}\label{subset}
\rho(v,U'):=  \int_{U'}\int_{U'}(v(x)-v(y))^2J(x-y)\ dxdy<\infty.
\end{equation}
Note that $\cD^{J}(U) \subset \cV^{J}(U')$ for any pair of open subsets $U,U' \subset \R^N$.

\begin{lemma}[see \cite{JW14}, Lemma 2.5]
\label{sec:linear-problem-tech-1}  
Let $U' \subset \R^N$ be an open set and $v,\phi \in \cV^{J}(U')$. Moreover, suppose that $\phi \equiv 0$ on $\R^N \setminus U$ for some open bounded subset $U \subset U'$ with $\dist(U,\R^N \setminus U')>0$. Then 
\begin{equation}
  \label{eq:finite-int}
\int_{\R^N} \int_{\R^N} |v(x)-v(y)| |\phi(x)-\phi(y)|k(|x-y|) \,dx dy < \infty,  
\end{equation}
and thus 
$$
\cJ(v,\phi) := \frac{1}{2} \int_{\R^N} \int_{\R^N} (v(x)-v(y)) (\phi(x)-\phi(y))k(|x-y|) \,dx dy  
$$
is well defined.   
\end{lemma}

 \begin{lemma}[see \cite{JW14}, Lemma 2.6]\label{3-cutoff}
If $U' \subset \R^N$ is open and $v\in \cV^J(U')$, then $v^\pm \in \cV^J(U')$ and $\rho(v^\pm,U') \le \rho(v,U')$. In addition, if $v\in \cV^J(\R^N)$, then $0\geq \cJ(v^+,v^-)> -\infty$.
\end{lemma}

\section{A linear problem associated with a hyperplane reflection}\label{mp}

In the following, we put $\cH$ as the set of all open affine half spaces in $\R^{N}$ and denote $Q_H(x)$ as the reflection of $x\in \R^N$ at $\partial H$ w.r.t. a given half space $H\in \cH$. Moreover, we put $\cH_0:=\{H\in \cH\;:\; 0\in \partial H\}$ and for $p\in \R^N$ we put
\begin{equation}\label{def:hp}
\cH_0(p):=\{H\in \cH_0\;:\; p\in H\}\quad \text{ and }\quad H_p:=\{x\in \R^N\;:\; x\cdot p>0\}\in \cH_0(p).
\end{equation}
For the sake of brevity, we sometimes write $\bar x$ in place of $Q_H(x)$ for $x \in \R^N$. Given $H\in \cH$, a function $v:\R^{N}\to \R^{N}$ is called  antisymmetric (with respect to $H$) if $v(\bar x)=-v(x)$ for $x \in \R^{N}$. 

\begin{bem}\label{symmetry-need}
We note that since $k$ satisfies (k) we have for any $H\in \cH$:
\begin{enumerate}
\item[(J1)] $k(|x-y|)=k(|\bar{x}-\bar{y}|)$ for all $x,y\in \R^N$ and
\item[(J2)] $k(|x-y|)\geq k(|x-\bar{y}|)$ for all $x,y\in H$
\end{enumerate}
Moreover, by (k) there is $r_0>0$ such that $k|_{(0,r_0))}$ is strictly decreasing. Thus we also have the following strict variant of (J2):
\begin{enumerate}
\item[(J3)] $\underset{y\in B_r(x)}\essinf\ (k(|x-y|) - k(|x- \bar y|))>0$ for all $x\in H$ and $r< \min\{r_0,\dist(x,\partial H)\}$.
 \end{enumerate}
 \end{bem}

The following is similar to but more general than Lemma 3.2 in \cite{JW14}

\begin{lemma}[see Lemma 4.7, \cite{J15}]\label{sec:linear-problem-tech}
Let $U' \subset \R^N$ be an open set with $Q_H(U')=U'$, and let $v \in \cV^{J}(U')$ be an antisymmetric function such that there is $\kappa\geq0$ with $v \ge -\kappa$ on $H \setminus U$ for some open bounded set $U \subset H$ with 
$\overline U \subset U'$.  Then the function $w:= 1_H\, (v+\kappa)^-$ is contained in $\cD^{J}(U)$ and satisfies 
\begin{equation}
  \label{eq:key-ineq}
\cJ(w,w) \le - \cJ(v,w) -2\kappa \int_H w(x) \kappa_H(x)\ dx.
\end{equation}
\end{lemma}

 \begin{proof}
 Since $v$ is antisymmetric we have by (J1), (J2) and the symmetry of $U'$
 \begin{align}
  &\rho(v,U') =\int_{U'\cap H}\int_{U'\cap H}(v(x)-v(y))^2 k(|x-y|)\ dxdy  \notag\\
 &\qquad\qquad+\int_{U'\setminus H}\int_{U'\setminus H}(v(x)-v(y))^2 k(|x-y|)\ dxdy+2\int_{U'\setminus H}\int_{U'\cap H}(v(x)-v(y))^2 k(|x-y|)\ dxdy  \notag\\
 &=2\int_{U'\cap H}\int_{U'\cap H}\Bigl[(v(x)-v(y))^2 k(|x-y|) + (v(x)+v(y))^2 k(|x-\bar y|)\Bigr]\ dxdy  \notag\\
 &\geq \int_{U'\cap H}\int_{U'\cap H} \Bigl[(v(x)-v(y))^2 k(|x- y|) + [(v(x)-v(y))^2 + (v(x)+v(y))^2] k(|x-\bar y|)\Bigr]\ dxdy  \notag\\ 
 &\geq \int_{U'\cap H}\int_{U'\cap H} \Bigl[(v(x)-v(y))^2 k(|x-y|)+ 2v^2(x) k(|x-\bar y|)\Bigr]\ dxdy  \notag\\ 
 &= \int_{U'}\int_{U'}(1_{H}v(x)-1_{H}v(y))^2 k(|x-y|)\ dxdy = \rho(1_{H}\,v ,U') \label{bound3}
 \end{align}
 Since $\rho(1_H v,U')=\rho(1_H v+\kappa,U')$ we have $1_Hv+\kappa\in \cV^J(U')$. And thus, since $v\geq -\kappa$ in $H\setminus U$ and $\kappa\geq 0$, we have $(1_Hv +\kappa)^-=1_H(v+\kappa)^-$. Hence by Lemma \ref{3-cutoff}  $w \in \cV^{J}(U')$. Since $w \equiv 0$ in $\R^N \setminus U$ and $\dist(U, \R^{N}\setminus U')>0$, $\cJ(v,w)$ is well defined and finite by Lemma~\ref{sec:linear-problem-tech-1}. To show  (\ref{eq:key-ineq}) we first note that with $\tilde{v}=v+\kappa$ we have
 $$
 [w+\tilde{v}]w= [ 1_H \tilde{v}^{\ +} + 1_{\R^N \setminus H} \tilde{v}]1_H \tilde{v}^{\ -} \equiv 0 \qquad  \text{on $\R^N$}
 $$
 and therefore  
 $$
 [w(x)-w(y)]^2 + [v(x)-v(y)] [w(x)-w(y)] = - \Bigl(w(x)[w(y)+\tilde{v}(y)] + w(y)[w(x)+\tilde{v}(x)]\Bigr)
 $$
 for $x,y \in \R^N$. Using this identity in the following together with the antisymmetry of $v$, the symmetry properties of $k$ and the fact that $w \equiv 0$ on $\R^N \setminus H$, we find that 
 \begin{align}
 \cJ(w,w) + &\cJ(v,w) = \cJ(w,w) + \cJ(\tilde{v},w)\notag\\
 &=-\int_{H} \int_{\R^N}w(x) [w(y)+\tilde{v}(y)] k(|x-y|)\,dy dx\\
  &=-\int_{H} \int_{\R^N}w(x) [1_H(y)\tilde{v}^{\ +}(y)  +1_{\R^N \setminus H} \tilde{v}(y)] k(|x-y|)\,dy dx\notag\\
 &=-\int_{H} \int_{H}w(x) [\tilde{v}^{\ +}(y)k(|x- y|) +(-v(y)+\kappa) k(|x-\bar y|)] \,dy dx \notag\\
 &=-\int_{H} \int_{H}w(x) [\tilde{v}^{\ +}(y)k(|x- y|)+ (-\tilde{v}(y)+2\kappa) k(|x-\bar y|)] \,dy dx\notag \\
&\leq -\int_{H} \int_{H}w(x) [\tilde{v}^{\ +}(y)\left(k(|x- y|)-k(|x-\bar y|)\right) +2\kappa k(|x-\bar y|)] \,dy dx \:\le\: 0, \label{ineq:lowerbound}
  \end{align}
 where in the last step we used (J2). Hence $\cJ(w,w)\leq -\cJ(v,w)$ and thus $\cJ(w,w)<\infty$. Since $w \equiv 0$ on $\R^N \setminus U$, it thus follows that $w \in \cD^{J}(U)$. Thus also the right hand side of (\ref{ineq:lowerbound}) is finite and hence (\ref{eq:key-ineq}) is true.
 \end{proof}

In order to implement the moving plane method and the rotating plane method, we have to deal with antisymmetric supersolutions of a class of linear problems. A related notion was introduced in \cite{JW13,JW14,J15}.

\begin{defi}\label{3-defi-anti}
Let $H\in \cH$, $U\subset H$ be an open set and let $c\in L^{\infty}(U)$.  We call an antisymmetric
function $v: \R^N \to \R^N$ an \textit{antisymmetric supersolution of the problem}
\begin{equation}\label{linear-prob}
Iv= c(x)v\quad \text{ in $U$,}\quad  v \equiv 0 \;\; \text{on $H \setminus U$.}
\end{equation}
if $v \in \cV^{J}(U')$ for some open bounded set $U' \subset \R^N$ with $Q_H(U')=U'$ and $\overline U \subset U'$,  $v\geq 0$ on $H\setminus U$, 
$ \liminf_{\substack{|x|\to\infty\\ x\in H}}v(x)\geq 0$ and
\begin{equation}\label{3-eq-sol2}
\cJ(v,\phi)\geq \int_{U}c(x)v(x)\varphi(x)\ dx \;\; \text{ for all $\varphi\in \cD^{J}(U)$, $\varphi\geq0$ with compact support in $\R^N$.} 
\end{equation}
\end{defi}

\begin{bem}\label{3-anti}
Let (k) be satisfied and assume $f:\R\to \R$ satisfies (F1). Then we have the following. If $u \in \cD^J(\Omega)$ is a solution of 
\[
 \qquad\left\{\begin{aligned}
  I u&=f(u)&& \text{ in $\Omega$;}\\
	u&\equiv 0&& \text{ on $\R^{N}\setminus \Omega$,}\\
 \end{aligned}\right.
\]
 then for $H\in \cH$ with $Q_H(\Omega\cap H)\subset \Omega$ we have that $v:=u\circ Q_H-u$ is an antisymmetric supersolution of (\ref{linear-prob}) with $U=\Omega\cap H$ and $c \in L^\infty(U)$ given for $x\in U$ by 
$$
c(x)= \left \{
  \begin{aligned}
  &\frac{f(u(\bar x))-f(u(x))}{v(x)}&&\qquad \text{if $v(x) \not= 0$;}\\
  &0 &&\qquad \text{if $v(x)= 0$.}
  \end{aligned}
\right.
$$
Indeed, since $u\in \cD^{J}(\Omega)\subset\cD^J(\R^N)$, we have $v \in \cD^{J}(\R^N)$ and thus $v \in \cV^{J}(U')$ for any open set $U' \subset \R^N$. Moreover, $\lim_{\substack{ |x|\to \infty \\ x\in H}} v(x)=0$ is satisfied since $u$ is satisfies $\lim\limits_{|x|\to \infty}u(x)=0$. Furthermore, if $\varphi\in \cD^{J}(U)$ with compact support in $\R^N$, then $\varphi\circ Q_H-\varphi\in \cD^J(\Omega)$ and this function has compact support in $\R^N$. Moreover, if $\varphi \ge 0$ then we have by (J1) 
\begin{align*}
&\cJ(v,\phi)= \cJ(u \circ Q_H - u,\phi)=\cJ(u,\phi \circ Q_H - \phi)= 
\int_{\R^N}f(u(x))[\varphi (Q_H(x)) - \varphi(x)]\,dx\\
&=\!\!\!\int_{ \R^{N}\setminus H}\!\!\!f(u(x))\varphi \circ Q_H\,dx - \int_{H}\!\!f(u(x))\varphi\,dx=\!\!\int_{H}[f(u(\bar x))-f(u(x))]\varphi(x)\,dx = \int_{H}c(x) v(x) \varphi(x)\,dx.
\end{align*}
The boundedness of $c$ follows from (F1).\\
Note that the same calculation holds if $\Omega$ is radial, $H\in \cH_0$ and $f:[0,\infty)\times \R\to \R$ satisfies (F1).
\end{bem} 

Next we present some variants of maximum principles for antisymmetric supersolutions of (\ref{linear-prob}). 

\begin{prop}[see \cite{JW14}, Proposition 3.5]\label{weak1}
Let $U \subset H$ be an open bounded set and let $c\in L^{\infty}(U)$ with $\|c^+\|_{L^\infty(U)} <\Lambda_1(U)$, where $\Lambda_1(U)$ is given in (\ref{l2-bound}).\\
Then every antisymmetric supersolution $v$ of (\ref{linear-prob}) in $U$  with $v\geq 0$ in $H\setminus U$ satisfies $v\geq 0$ a.e. in $H$. 
\end{prop}

\begin{lemma}[see \cite{J15}]\label{weak2}
Let $H\in \cH$ and let $U\subset H$ be an open set. Let $c\in L^{\infty}(U)$ with $c\leq 0$ in $U$. If $v$ is an antisymmetric supersolution of (\ref{linear-prob}) in $U$, then $v\geq 0$ a.e. in $H$.
\end{lemma}
\begin{proof}
Since $\liminf_{\substack{|x|\to\infty\\ x\in H}}v(x)\geq 0$ we have that for every $\epsilon>0$ there is $R>0$ such that $v\geq -\epsilon$ on $H\setminus B_{R}(0)$. Put $\varphi=1_H(v+\epsilon)^-$. Then $\tilde{v}\in \cD^{J}(U\cap B_{R}(0))$ by Lemma \ref{sec:linear-problem-tech} and we have
\begin{align*}
 0\geq \cJ(\varphi,\varphi)\leq -\cJ(v,\varphi)\leq  -\int_{U}c(x)v(x)\varphi(x)\ dx\leq 0,
\end{align*}
since $c\leq 0$ in $U$ and $v(x)\varphi(x)=-\varphi^2(x)-\epsilon\varphi(x)\leq 0$ for $x\in U$. Thus $0=\cJ(\varphi,\varphi)\geq \Lambda_1(B_{R}(0)\cap U)\|\varphi\|^2_{L^2(H)}$ for any $\epsilon>0$. This proves the claim.
\end{proof}

\begin{prop}\label{weak3}
Let $H\in \cH$ and let $U\subset H$ be an open set. Let $c\in L^{\infty}(U)$ and assume there is $B\subset \R^N$ such that $c\leq 0$ in $ U\setminus B$. Then there is $d>0$ independent of $H$ such that the following is true: If $v$ is an antisymmetric supersolution of (\ref{linear-prob}) in $U$  with 
$$
v(x)\geq 0 \qquad \text{for $x\in B\cap U$ with $\dist(x,\R^{N}\setminus H)\geq d$,}
$$
then $v\geq 0$ a.e. in $H$.
\end{prop}

\begin{proof}
Denote $c_{\infty}:=\|c\|_{L^{\infty}(U)}$. By translation and rotation using (k) we have for $\lambda>0$
\begin{equation}
 \inf_{\substack{x\in H \\\dist(x,\R^N\setminus H)<\lambda}} \;\;\;\kappa_H(x)= \inf_{\substack{x\in H \\\dist(x,\R^N\setminus H)<\lambda}} \;\;\; \int_{\R^N\setminus H}k(|x-y|)\ dy\geq \int_{\{y_1>\lambda\}}k(|y|)\ dy
\end{equation}
Since (k) implies $\int_{\{y_1>0\}}k(|y|)\ dy=\infty$, there is $d>0$ depending on $N,k$ and $c_{\infty}$ such that 
\begin{equation}
\inf_{\substack{x\in H \\ \dist(x,\R^N\setminus H)<d}} \;\;\;\int_{\R^N\setminus H}k(|x-y|)\ dy> c_\infty. \label{choiceofdelta}
\end{equation}
Next let $\delta>0$ and denote  $\varphi_{\delta}(x)=(v+\delta)^{-}(x)1_{H}(x)$ for $x\in \R^N$. Thus by Lemma \ref{sec:linear-problem-tech} we have $\varphi_{\delta}\in \cD^J(U)$ with compact support in $\overline{H}$ since $\liminf\limits_{\substack{|x|\to \infty\\x\in H}}v(x)=0$. Testing (\ref{linear-prob}) with $\phi_{\delta}$ we get
\begin{equation}
\cJ(v,\phi_{\delta})\geq\int_{H}c(x)v(x)\phi_{\delta}(x)\ dx=-\int_{H}c(x)\phi_{\delta}^2(x)\ dx-\delta\int_{H}c(x)\phi_{\delta}(x)\ dx.\label{komisch1}
\end{equation}

Next denote $A:=\{x\in U\cap B\;:\; \dist(x,\R^N\setminus H) <d\}$ and note that
\begin{align*}
\{x\in A\;:\;&\text{$\phi_{\delta}(x)>0$ and $c(x)>0$}\}=\{x\in H\;:\;\text{$\phi_{\delta}(x)>0$ and $c(x)>0$}\}.
\end{align*}
Since $c(x)\leq 0$ for $x\in H\setminus (A\cup B)$ we have in (\ref{komisch1})
\begin{align}
\cJ(v,\phi_{\delta})& \geq - c_\infty \int_{A}\phi_{\delta}(x)(\phi_{\delta}(x) + \delta )\ dx.\label{komisch3}            
\end{align}
Moreover, by Lemma \ref{sec:linear-problem-tech} we have
\begin{align}
\cJ&(v,\phi_{\delta})\leq -\cJ(\phi_{\delta},\phi_{\delta}) -2\delta \int_{H}\int_{H}\phi_{\delta}(x)k(|x-\bar{y}|)\ dydx\notag\\
&\leq -\int_{H}\phi_{\delta}^2(x)\int_{\R^{N}\setminus H} k(|x-y|)\ dy\ dx -2\delta \int_{H} \phi_{\delta}(x)\int_{\R^{N}\setminus H} k(|x-y|)\ dy \ dx\notag\\
&\leq -\int_{A}\phi_{\delta}(x)\left(\phi_{\delta}(x)+2\delta\right)\int_{\R^{N}\setminus H} k(|x-y|)\ dy\ dx\le -c_{\infty}\int_{A}\phi_{\delta}(x)\left(\phi_{\delta}(x)+2\delta\right)\ dx.\label{komisch4}
\end{align}
Here we used that $\supp(\phi_{\delta})\subset H$. Combining (\ref{komisch3}) and (\ref{komisch4}) we have a contradiction unless $\varphi_{\delta}\equiv 0$ a.e. in $A$. Since $\delta>0$ was chosen arbitrary we have $v\geq 0$ in $B\cap U$. Since $v$ is thus an antisymmetric supersolution of (\ref{linear-prob}) in $U\setminus B$ and $c\leq 0$ in $U\setminus B$, Lemma \ref{weak2} gives $v\geq 0$ in $H$ as claimed.
\end{proof}

\begin{prop}[see \cite{JW14}, Proposition 3.6]\label{strong1}
Let $H\in \cH$ and let $U\subset H$ be an open domain. Let $c\in L^{\infty}(U)$ and let $v$ be an antisymmetric supersolution of (\ref{linear-prob}) in $U$ satisfying $v\geq 0$ a.e. in $H$. Then either $v\equiv 0$ in a neighborhood of $\overline{U}$ or 
\[
\underset{K}\essinf\ v>0\qquad \text{ for every compact $K\subset U$.}
\]
\end{prop}

\begin{prop}[see \cite{JW14}, Proposition 5.1]\label{strong2}
Assume additionally that $k:(0,\infty)\to(0,\infty)$ is strictly decreasing. Let $H\in \cH$ and $U\subset H$ be an open set. Let $c\in L^{\infty}(U)$ and let $v$ be an antisymmetric supersolution of (\ref{linear-prob}) in $U$ satisfying $v\geq 0$ a.e. in $H$. Then either $v\equiv 0$ in $\R^N$ or 
\[
\underset{K}\essinf\ v>0\qquad \text{ for every compact $K\subset U$.}
\]
\end{prop}

\begin{bem}\label{strongmp}
Note that Proposition \ref{strong1} and \ref{strong2} also hold for supersolutions, that is $v\in \cD^J(U)$ satisfies
\[
 \cJ(v,\varphi)\geq \int_{U} c(x)v(x)\varphi(x)\ dx\quad\text{ for all $\varphi\in \cD^J(U)$ with compact support in $\R^N$}
\]
and $v\geq 0$ on $\R^N\setminus U$ (see also \cite{J15}).
\end{bem}


\section{Foliated Schwarz symmetry for nonlocal boundary value problems}\label{foliatedsymmgoal}

In addition to the definitions at the beginnings of Section \ref{setup} and \ref{mp} we need the following. Given $H\in \cH$ we define the polarization of $u:\R^N\to \R$ w.r.t. $H$ as
\begin{equation}
 u_H(x):=\left\{\begin{aligned}
&\max\{u(x),u(Q_H(x))\}&&\text{ if $x\in H$;}\\
&\min\{u(x),u(Q_H(x))\}&&\text{ if $x\in \R^{N}\setminus H$.}
\end{aligned}\right.
\end{equation} 
We say that $H$ is dominant for $u$, if $u= u_H$ on $H$.

In this part we will work with the rotating plane method. The following characterization of foliated Schwarz symmetry for continuous functions will be helpful.

\begin{prop}[Proposition 3.3, \cite{SW13}]\label{foliatedchar}
Let $\Omega\subset \R^N$ be an open radial set and let $U$ be a set of functions $u:\R^N\to \R$, which are continuous in $\Omega$. Moreover, fix
\[
M:=\{e\in S^{N-1}\;:\; u=u_{H_p}\ \text{ for all $x\in \Omega \cap H_p$ and $u\in U$}\}
\]
and assume that there is $e_0\in M$ such that the following is true:\\
If for all two dimensional subspaces $P\subset \R^N$ containing $e_0$ there are two different points $p_1$, $p_2$ in the same connected component of $M\cap P$ such that $u=u\circ Q_{H_{p_1}}$ and $u=u\circ Q_{H_{p_2}}$ for every $u\in U$, then there is $p\in S^{N-1}$ such that for every $u\in U$ and every connected component $D\subset \Omega$ the functions $\left.u\right|_D$ are foliated Schwarz symmetric with respect to $p$.
\end{prop}

\begin{proof}[Proof of Theorem \ref{schwarzsymm}]
To fix our notation for the proof of Theorem \ref{schwarzsymm} we denote for $e\in S^{N-1}$, $H_e\in\cH_0$ as in \ref{def:hp}. Moreover, we put
\[
\sigma_e:=Q_{H_e}\qquad \text{ and }\qquad \Omega_e:=\Omega \cap H_e.
\]
Note that then the function $v_e:=u\circ \sigma_e-u$ is a continuous antisymmetric supersolution of (see Remark \ref{3-anti})
\begin{equation}\label{antisym1}
Iv=c_e(x)v\quad \text{ in $\Omega_e$;}  \qquad v \equiv 0\quad\text{on $H_e\setminus \Omega_e$,}
\end{equation}
where for $x\in \Omega_e$ we put
\[
c_e(x):=\left\{\begin{aligned} \frac{f(|x|,u\circ \sigma_e(x))-f(|x|,u(x))}{u\circ \sigma_e(x)-u(x)}&&\text{ if $u\circ \sigma_e(x)\neq u(x)$;}\\
0&&\text{ if $u\circ \sigma_e(x)= u(x)$.}\end{aligned}\right.
\]
Moreover, since we assume (F1) and $u$ is bounded there is $c_{\infty}>0$ such that
\[
\sup_{e\in S^{N-1}}\|c_e\|_{L^{\infty}(\Omega_e)}\leq c_\infty.
\]

 We will use Proposition \ref{foliatedchar} to prove the statement. For this put
\begin{equation}\label{def:m}
M:=\{e\in S^{N-1}\;:\; v_e\geq 0 \;\text{ in $\Omega_e$}\}.
\end{equation}
Let $e_1\in S^{N-1}$ be given by (U1). Note that by (U1) we have that $H_{e_1}$ is dominant for $u$ and thus $e_1\in M$. Since moreover (U1) implies $v_{e_1}\not\equiv 0$ in $\Omega_{e_1}$, Proposition \ref{strong1} -- in case $\Omega$ is not connected -- or Proposition \ref{strong2} -- assuming $k$ is additionally strictly decreasing -- give $v_{e_1}>0$ in $\Omega_{e_1}$. We then have to show 
\begin{enumerate}
\item[(S)] For all two dimensional subspaces $P\subset \R^N$ containing $e_1$ there are two different points $p_1$, $p_2$ in the same connected component of $M\cap P$ such that $u\equiv u\circ Q_{H_{p_1}}$ and $u\equiv u\circ Q_{H_{p_2}}$
\end{enumerate}

To show (S), let $A:(-\pi,\pi]\times \R^N\to \R^{N}$, $(\varphi,x)\mapsto A(\varphi)x$ be a rotation of $\R^N$ of angle $\varphi$ and consider the set $S:=\{A(\varphi)e_1 \;:\;\varphi\in[-\pi,\pi)\}$ and write $e^{\varphi}=A(\varphi)e_1$ for the elements in $S$.

By continuity of $v_{e_1}$ we may choose $\epsilon>0$ small, such that there is $K\subset \Omega_{e^\varphi}$ with $\dist(K,\R^{N}\setminus H_{e^{\varphi}})>0$ and $v_{e^\varphi}\geq0$ in $K$ for all ${e^\varphi}\in S^{N-1}$ with $\varphi\in(-\epsilon,\epsilon)$. Moreover, we can choose $K$ such that additionally we have $|\Omega_{e^\varphi}\setminus K|<\delta(\epsilon)$ with $\delta(\epsilon)\to 0$ for $\epsilon\to 0$. Thus we may assume -- making $\epsilon$ even smaller -- that we have $\Lambda_1(\Omega_{e^\varphi}\setminus K)>c_{\infty}$ for all $\varphi\in (-\epsilon,\epsilon)$, which is possible due to (\ref{l2-bound}) since $\inf_{\substack{B\subset \R^{N}\\|B|=m}}\Lambda_1(B)\to \infty$ for $m\to 0$. An application of Proposition \ref{weak1} gives $v_{e^\varphi}\geq 0$ in $\Omega_{e^{\varphi}}$ so that
\[
\{A(\varphi)e_1\;:\; \varphi\in(-\epsilon,\epsilon)\}\subset M.
\]
For the next step we put
\begin{equation}\label{phipm}
\varphi^{+}:=\sup\{\varphi\;:\; e^{\varphi}\in M\}\quad\text{ and }\quad\varphi^{-}:=\inf\{\varphi\;:\; e^{\varphi}\in M\}.
\end{equation}
Note that $\varphi^+\in [\epsilon,\pi-\epsilon]$ and $\varphi^-\in [-\pi+\epsilon,-\epsilon]$. Moreover, by continuity we have $v_{e^{\varphi^\pm}}\geq 0$ in $\Omega_{e^{\varphi^{\pm}}}$ thus by Proposition \ref{strong1} or Proposition \ref{strong2} we have either
\begin{align*}
\text{Case 1:}\;\;&\text{ $v_{e^{\varphi^{+}}}\not\equiv 0$ in $\Omega_{e^{\varphi^+}}$ or $v_{e^{\varphi^{-}}}\not\equiv 0$ in $\Omega_{e^{\varphi^-}}$.}\\
\text{Case 2:}\;\;&\text{ $v_{e^{\varphi^{+}}}\equiv 0$ in $\Omega_{e^{\varphi^+}}$ and $v_{e^{\varphi^{-}}}\equiv 0$ in $\Omega_{e^{\varphi^-}}$.}
\end{align*}
Note that in Case 1 we have by Proposition \ref{strong1} or Proposition \ref{strong2} $v_{e^{\varphi^{+}}}> 0$ in $\Omega_{e^{\varphi^+}}$ or $v_{e^{\varphi^{-}}}> 0$ in $\Omega_{e^{\varphi^-}}$ and with the arguments as in the beginning this leads to a contradiction of the definition of $\varphi^+$ or $\varphi^-$. Thus we must be in Case 2. Since $e^{\varphi^+}\neq e^{\varphi^-}$ are in the same connected component of $M$ by construction and, moreover, since we have choose $A$ arbitrary (S) follows and so the proof is finished using Proposition \ref{foliatedchar}.
\end{proof}

\begin{proof}[Proof of Theorem \ref{schwarzsymm2}]
Let $\Omega$ be as stated, let (F2) be satisfied and additionally, if $\Omega$ is not connected, we assume that $k:(0,\infty)\to(0,\infty)$ is strictly decreasing.  By (F1) and since $u$ is bounded we have for any $e\in S^{N-1}$ that
\[
c_e:\Omega\to \R, \quad c_e(x):=\left\{\begin{aligned} \frac{f(x,u\circ \sigma_e(x))-f(x,u(x))}{u\circ \sigma_e(x)-u(x)}&&\text{ if $u\circ \sigma_e(x)\neq u(x)$;}\\
0&&\text{ if $u\circ \sigma_e(x)= u(x)$}\end{aligned}\right.
\] 
satisfies $c_e\in L^{\infty}(\Omega)$ with
\[
 c_{\infty}=\sup_{e\in S^{N-1}} \|c_e\|_{L^{\infty}(\Omega)}<\infty.
\]
Moreover, by (U2) we have $\lim\limits_{r\to\infty}\sup\limits_{|x|=r} u(x)=0$. Thus there is -- using (F2) --  a radius $\rho>0$ independent of $e$ such that
\[
\sup_{e\in S^{N-1}}c_e(x)\leq 0\qquad\text{ for all $x\in \Omega\setminus B_{\rho}(0)$.}
\]
We may proceed similarly as in the proof of Theorem \ref{schwarzsymm}, i.e. let $e_1$ be given by (U1), consider the statement (S), let $A(\cdot)$ be a rotation of $\R^N$ and fix $e^{\varphi}:=A(\varphi)e_1$. Note that we have as before $v_{e_{1}}>0$ in $\Omega_{e_1}$. Hence using the continuity of $u$ we may fix $K\subset\subset \Omega_{e^{\varphi}}\cap B_{\rho}(0)$ such that $| \Omega_{e^{\varphi}}\cap B_{\rho}(0)\setminus K|$ is small and $v_{e^{\varphi}}\geq 0$ in $K$ for $\varphi\in(-\epsilon,\epsilon)$ with $\epsilon$ small enough. As in the proof of Theorem \ref{schwarzsymm}, using Proposition \ref{weak3} with $U=\Omega_{e^{\varphi}}$ and $B=\Omega_{e^{\varphi}}\cap B_{\rho}(0)$ instead of Proposition \ref{weak1}, we find $\varphi^+$ and $\varphi^-$ as in (\ref{phipm}) such that $\partial H_{e^{\varphi^+}}$ and $\partial H_{e^{\varphi^-}}$ are two symmetry hyperplanes of $u$. We thus conclude with Proposition \ref{foliatedchar} that there is for each connected component $D$ of $\Omega$ the same $e_0\in S^{N-1}$ as claimed for which $\left.u\right|_D$ is foliated Schwarz symmetric.
\end{proof}

\section{Foliated Schwarz symmetry for global minimizers}\label{functional}

Next we will give an application of Theorem \ref{schwarzsymm} and Theorem \ref{schwarzsymm2}. Let $\Omega\subset \R^N$, $N\geq 2$ be any radial open set and let $k:(0,\infty)\to(0,\infty)$ be a strictly decreasing function which satisfies (k) and
\begin{enumerate}
 \item[(k2)] there is $c_1,r_0>0$ and $s\in(0,1)$ such that
\[
 k(r)\geq c_1 r^{-1-2s}\quad\text{ for $r\in(0,r_0]$.}
\]
\end{enumerate}
Then $\cD^J(\Omega)\subset \cH^s_0(\Omega)$ (see e.g. \cite[Lemma 5.14]{J15}) and for $q\in[2,2_s^{\ast}]$, $2_s^{\ast}=\frac{2N}{N-2s}$ it follows that there is $C>0$ such that (see e.g. \cite[Theorem 6.7]{NPV11}) 
\[
\|u\|_{L^{q}(\Omega)}\leq  C\left(\cJ(u,u)+\|u\|^2_{L^2(\R^N)}\right) \quad\text{ for all $u\in \cD^J(\Omega)$.}
\]
 We will consider global minimizers of the functional
\begin{equation}\label{defifunc}
\begin{aligned}
K:\ &\cD^{J}(\Omega)\to \R,&& K[u]=\frac{1}{2}\cJ(u,u)-\int_{\Omega} F(|x|,u(x))\ dx.
\end{aligned}
\end{equation}
Here 
\[
F:[0,\infty)\times \R\to\R,\qquad F(r,u)=\int_0^{u} f(r,\tau)\ d\tau,
\]
where $f:[0,\infty)\times \R\to\R$ satisfies (F1), (F2) and there are $a_1,a_2$ and $q\in[2,2_s^{\ast}]$ such that
\[
|f(r,u)|\leq a_1|u|+a_2|u|^{q-1} \qquad\text{ for all $r\geq 0$, $u\in \R$.} 
\]
Note that thus $F$ satisfies
\[
 |F(r,u)|\leq a_1 |u|^2+a_2|u|^q\qquad \text{ for all $r\geq 0$, $u\in \R$}
\]
and hence the functional $K$ is well-defined on $\cH^s_0(\Omega)$. Moreover, as described in \cite[Chapter 3]{SV12} (see also \cite{SV13}) $K$ is Fr\'echet differentiable and for $u,\varphi\in\cD^J(\Omega)$ we have
\[
 \langle K'[u],\varphi\rangle =\cJ(u,\varphi)-\int_{\Omega}f(|x|,u(x))\varphi(x)\ dx.
\]
 This gives that global minimizers are solutions of problem (P'). The following proves Theorem \ref{result4}.

\begin{satz}\label{result5}
 Let $\Omega\subset\R^N$ be an open radial set and let $F$ be given as above w.r.t. $f$ and $q$ as above. Moreover, assume (k) and (k2) hold and let either
\begin{enumerate}
 \item $u\in \cD^J(\Omega)$ be a continuous bounded global minimizer of $K$ or
\item $u\in \cD^J(\Omega)$ be a continuous bounded minimizer of $K$ subject to $\|u\|^q=1$.
\end{enumerate}
If in addition $u$ satisfies (U2) (see p. \pageref{u2}), then $u$ is foliated Schwarz symmetric.
\end{satz}

\begin{bem}
 Note that Theorem \ref{result5} is a special case of Theorem \ref{result4}. 
\end{bem}

For the proof we will need the following inequality concerning the polarization of a function (see \cite[Theorem 2]{B92} or \cite[Proposition 8]{VW04}): If (k) is satisfied, then we have for every $H\in\cH$
\begin{equation}\label{pol}
 \cJ(u_H,u_H)\leq \cJ(u,u) \qquad\text{ for all measurable $u:\R^N\to\R$ with $\cJ(u,u)<\infty$.}
\end{equation}

With this we give now the

\begin{proof}[Proof of Theorem \ref{result5}]
First let $u$ be a continuous bounded global minimizer of the functional $K$ and assume $u$ satisfies (U2). Then $u$ solves
\[
 (P')\qquad\left\{\begin{aligned}
  Iu&=f(|x|,u)&& \text{ in $\Omega$}\\
u&=0,&& \text{ on $\R^N\setminus \Omega$}\\
\lim\limits_{|x|\to\infty}u(x)&=0
 \end{aligned}\right.
\]
If $u$ is radial symmetric, then $u$ is in particular foliated Schwarz symmetric. Hence assume that $u$ is not radial, then there is $x_0\in \Omega$ such that $u(x_0)>u(-x_0)$. Let $H\in \cH_0(x_0)$ such that $Q_H(x_0)=-x_0$ and note that with (\ref{pol}) the polarization $u_H$ of $u$ w.r.t. $H$ is also a minimizer of $K$. Note that we have $v=u_H-u$ is antisymmetric w.r.t. $H$ with $v\geq 0$ on $H$. Since $u$ and $u_H$ are minimizers we have that both functions solve $(P')$, i.e. for all $\varphi\in \cD^J(\Omega\cap H)$, $\varphi\geq0$ we have
\[
\cJ(v,\varphi)=\int_{\Omega\cap H}\left(f(|x|,u_H(x))\varphi(x)-f(|x|,u(x))\right)\varphi(x)\ dx=\int_{\Omega\cap H}c(x)v(x)\varphi(x)\ dx,
\]
where for $x\in \Omega\cap H$ we have
\[
 c(x)=\int_{0}^{1} \partial_uf(|x|,u(x)+s(u_H(x)-u(x)))\ ds.
\]
Note that $c\in L^{\infty}(\Omega\cap H)$ since we assumed that $u$ is bounded. Thus we have that $v$ is an antisymmetric supersolution of $Iv=c(x) v$ in $\Omega\cap H$ with $v\geq 0$ on $H$. By Proposition \ref{strong1} we have that $v\equiv 0$ in $\R^N$ or $v>0$ in $\Omega\cap H$. Since $v$ is continuous and since $v(x_0)=0$ by assumption, we conclude $v\equiv 0$ in $\R^N$. In particular, we have $u_H=u$ in $\R^N$, i.e. $H$ is dominant for $u$. Since $u\not \equiv u\circ Q_H$ on $H$ we have that $u$ satisfies (U1) and thus an application of Theorem \ref{schwarzsymm} -- if $\Omega$ is bounded -- or Theorem \ref{schwarzsymm2} -- if $\Omega$ is unbounded -- finishes the proof of 1.\\
To see 2. note that then $u$ solves 
\[
 (P'_c)\qquad\left\{\begin{aligned}
  Iu&=f(|x|,u)+ |u|^{q-2}u&& \text{ in $\Omega$}\\
u&=0,&& \text{ on $\R^N\setminus \Omega$}\\
\lim\limits_{|x|\to\infty}u(x)&=0
 \end{aligned}\right..
\]
Moreover, using (\ref{pol}) and the fact that for any $H\in \cH_0$ we have
\[
 \|u_H\|_{L^{q}(\Omega)}=\|u\|_{L^{q}(\Omega)},
\]
it follows that the polarization $u_H$ of $u$ is also a minimizer of the functional $K$ under the same constraint. Hence $u_H$ also solves $(P'_c)$. Now 2. follows with the same arguments as for 1.
\end{proof}

\section{Proof of Theorem \ref{result1}}\label{mr}

Since $k$ satisfies (k) we have that (J1) -- (J3) in Remark \ref{symmetry-need} hold for any $H\in \cH$. Moreover, we assume that $f:\R\to\R$ satisfies (F1) and (F2). Let $u\in \cD^J(\R^N)$ be a nonnegative bounded continuous solution of $(R)$ which satisfies 
\begin{equation}\label{far}
\lim_{|x|\to\infty}u(x)=0.
\end{equation}

To prove Theorem \ref{result1} we will fix some $e\in S^{N-1}$ and apply the moving plane method with respect to reflections at $H_{\lambda}:=\{x\cdot e>\lambda\}\in \cH$, $\lambda\in \R$. Denote for $\lambda\in \R$: $T_{\lambda}:=\partial H_{\lambda}$, $Q_{\lambda}:=Q_{H_{\lambda}}$ and for any function $z:\R^{N}\to \R$ define by $z_{\lambda}(x):=z(Q_{\lambda}(x))$ the reflected function. Furthermore we will denote $V_\lambda z=z_{\lambda}-z$, the difference between $z$ reflected and the original $z$. Note that $V_{\lambda}z$ is antisymmetric w.r.t. $H_{\lambda}$.

By reflecting problem $(R)$ we will get that $u_{\lambda}$ solves for any $\lambda\in \R$ again
\begin{equation*}
 (R)\qquad
\left\{
 \begin{aligned}
            Iu_{\lambda}&= f(u_{\lambda})&\qquad \text{ in $\R^{N}$},\\
                        \lim\limits_{|x|\to\infty}\  u(x^{\lambda})&= 0, &\qquad \text{}\\
 \end{aligned}
\right.
\end{equation*}

Put $v(x):=V_{\lambda}u(x)=u(x^{\lambda})-u(x)$, then $v$ is a continuous antisymmetric supersolution of  (see Remark \ref{3-anti})
\[
(R) \qquad Iv= c(x)v\quad \text{ in $ H_{\lambda}$,}\quad \lim\limits_{\substack{|x|\to\infty\\x\in H_{\lambda}}}\ v(x)= 0,
\]
where for $x\in H_{\lambda}$ we put
\begin{equation*}
c(x):= \left\{ \begin{aligned}\frac{f(u_{\lambda}(x))-f(u(x))}{u_{\lambda}(x)-u(x)} && \text{if}\;\;  u_{\lambda}(x)\neq u(x); \\ 
                0 && \text{if}\;\; u_{\lambda}(x)=u(x).
               \end{aligned}
\right.
\end{equation*}
Note that by the assumption (F1) and since $u$ is bounded, we have
$$
c_\infty:= \sup_{\substack{\lambda \in \R \\ e\in S^1}}\|c_{\lambda,e}\|_{L^\infty(H_{\lambda,e})} < \infty.
$$
Furthermore, by (F2) and (\ref{far}), we can pick $\rho>0$ large enough such that for any $\lambda\in \R$ 
\begin{equation}\label{decay}
c(x)\le 0 \text{ for all $x\in \R^{N}$ such that $|x|\geq \rho$ and $|x^{\lambda}|\geq\rho$.}
\end{equation}
Denote 
\begin{equation}\label{defi-g}
G_{\lambda}:=B_{\rho}(0)\cup Q_{\lambda}(B_{\rho}(0)).
\end{equation}

Assume that $u$ is nontrivial and consider the statement
\[
 (S)_{\lambda}\qquad V_{\lambda}u> 0 \text{ in } H_{\lambda}.
\]
We will show the following steps to prove the statement.
\begin{enumerate}
 \item[Step 1] $(S)_{\lambda}$ holds for $\lambda$ sufficiently large.
\item[Step 2] Define $\lambda_{\infty}:=\inf \{\mu\ : \ (S)_{\lambda} \text{ holds for all } \lambda\geq \mu\},$ and prove
$$
\lambda_{\infty}>-\infty \qquad\text{ and }\qquad V_{\lambda_{\infty}}u\equiv 0 \quad \text{ on $\R^{N}$.}
$$
\end{enumerate}

Note that Step 1 and Step 2 imply that for all $e\in S^1$ there is a hyperplane $T^e$ perpendicular to $e$ and such that $u$ is symmetric with respect to $T^e$ and strictly decreasing in direction $e$. In particular, considering hyperplanes $T^{e_i}$ corresponding to the coordinate vectors $e_i$, we have that $u$ is symmetric with respect to $T^{e_i}$ for $i=1,\dots,N$ and strictly decreasing in all coordinate directions. Consequently, $u$ is also symmetric with respect to reflection at the unique intersection point $z_0$ of $T^{e_1},\dots,T^{e_N}$. It is then easy to see that there is $z_0 \in T^{e}$ for all $e \in S^{N-1}$, and this implies that $u$ is radial up to translation about the same point $z_0$ (for details we refer to the survey \cite{T}).\\
Moreover, we have that $u(\cdot-z_0)$ is strictly decreasing in its radial direction.

We will need the following.

\begin{lemma}\label{pos}
 We have $u>0$ in $\R^N$.
\end{lemma}
\begin{proof}
 By continuity and since $u\geq 0$, $u\not\equiv 0$ there is an open set $D\subset \R^N$  with $\inf_D u>0$. By (F1) and (F2) we can linearize problem (P) so that $u$ solves
\[
 Iu=d(x)u \quad \text{ in $\R^N$,} \quad \lim\limits_{|x|\to\infty} u(x)=0,
\]
where for $x\in \R^N$ we have 
\begin{equation*}
d(x):= \left\{ \begin{aligned}\frac{f(u(x))}{u(x)} && \text{if}\;\; u(x)\neq 0; \\ 
                0 &&\text{if}\;\;  u(x)=0.
               \end{aligned}
\right.
\end{equation*}
Proposition \ref{strong1} in combination with Remark \ref{strongmp} give $u>0$ in $\R^N$ as claimed.
\end{proof}

\noindent\textbf{Step 1: Large $\lambda$}

\begin{lemma}\label{start1}
 There exists $\lambda_{1}\in\R$ such that $(S)_{\lambda}$ holds for all $\lambda>\lambda_{1}$
\end{lemma}
\begin{proof}
 Note that for $\lambda$ sufficiently large we have
\[
 H_{\lambda}\cap G_{\lambda}=Q_{\lambda}(B_{\rho}(0)).
\]
Denote $\kappa=\inf_{B_{\rho}(0)}u$. Then $\kappa>0$ by Lemma \ref{pos}. Moreover, since  $\lim\limits_{|x|\to\infty}\ u(x)=0$, we have for $\lambda$ possibly larger
\[
 u(y)\leq \frac{\kappa}{2} \text{ for $y \in Q_{\lambda}(B_{\rho}(0))$,}
\]
since $\inf\{|x|\;:\;x\in Q_{\lambda}(B_{\rho}(0))\} \to \infty$ as $\lambda\to\infty$. Take $\lambda_{1}<\infty$ as the first value such that the above holds and note that thus for any $\lambda>\lambda_{1}$ we have
\[
 u(x)-u(x^{\lambda})\ge \frac{\kappa}{2}\text{ for $x \in B_{\rho}(0)$,}
\]
which is equivalent to
\[
 V_{\lambda}u(x)=u(x^{\lambda})-u(x)\ge \frac{\kappa}{2}\text{ for $x \in Q_{\lambda}(B_{\rho}(0))$.}
\]
Note that we can apply Proposition \ref{weak3}, since $c\leq 0$ in $\R^N\setminus G_\lambda$. Thus $V_{\lambda}u(x)\geq0$ in $H_{\lambda}$ for $\lambda\ge \lambda_1$. Hence by Proposition \ref{strong1} we have that $(S)_{\lambda}$ holds for any $\lambda\ge \lambda_{1}$.
\end{proof}

\noindent\textbf{Step 2: $\lambda=\lambda_{\infty}$}

We will fix $\lambda_{1}$ given by Lemma \ref{start1} and let $\lambda_{\infty}=\inf \{\mu\ : \ (S)_{\lambda} \text{ holds for all } \lambda\geq \mu\}$ be defined as above.
\begin{lemma}\label{step2}
 The following statements hold:
\begin{enumerate}
 \item[(i)] $-\infty<\lambda_{\infty}\leq \lambda_{1}$.
\item[(ii)] We have $V_{\lambda_{\infty}}u\equiv0$ on $\R^N$.
\end{enumerate}
\end{lemma}
\begin{proof}
(i) This follows with the same argument as in Lemma \ref{start1} applied to reflections at $H_{\lambda,-e}$ for $\lambda$. Thus $(S)_{\lambda}$ does not hold for $\lambda$ sufficiently negative. 

(ii) If there is $x_0\in H_{\lambda_{\infty}}$ with $V_{\lambda_{\infty}}u(x_0)=0$, then by Proposition \ref{strong1} we have $V_{\lambda_{\infty}}u\equiv 0$ on $\R^N$ as claimed. Thus assume $V_{\lambda_{\infty}}u>0$ on $H_{\lambda_{\infty}}$. Let $d$ be given as in Proposition \ref{weak3} and let $\lambda\in(\lambda_{\infty}-d,\lambda_{\infty})$. Then for $\lambda$ sufficiently close to $\lambda_{\infty}$ we have by continuity that $V_{\lambda}u\geq 0$ in $G_{\lambda}\cap H_{\lambda_{\infty}}$. Since $c\leq 0$ in $H_{\lambda}\setminus G_{\lambda}$, an application of Proposition \ref{weak3} gives $V_{\lambda}u\geq 0$ in $H_{\lambda}$ and thus $(S)_\lambda$ holds for all $\lambda<\lambda_{\infty}$ with $\lambda$ sufficiently close to $\lambda_{\infty}$ by Proposition \ref{strong1}. This is a contradiction to the definition of $\lambda_{\infty}$ and thus $V_{\lambda_{\infty}}u\equiv0$ on $\R^N$ as claimed.
\end{proof}

\bibliographystyle{amsplain}

\end{document}